\documentclass[12pt]{article}%
\usepackage{amssymb}
\usepackage{amsmath}
\usepackage{amsthm}
\usepackage{graphicx}
\usepackage{graphics}
\usepackage{epstopdf}
\usepackage{amsfonts}
\usepackage{color}
\usepackage{float}
\usepackage{hyperref}
\newtheorem{theorem}{Theorem}

\newtheorem{lemma}{Lemma}

\newcommand{\irt}{{\textsf{IRT}}}

\begin{document}
\title{On the maximum number of isosceles \\right triangles in a finite point set}
\author{Bernardo M. \'Abrego, Silvia Fern\'andez-Merchant,\\
and David B. Roberts\\{\small Department of Mathematics}\\{\small California State University, Northridge,}\\
{\small 18111 Nordhoff St, Northridge, CA
91330-8313.}\\\footnotesize{email:\texttt{\{bernardo.abrego, silvia.fernandez\}@csun.edu}}\\
\footnotesize{\texttt{david.roberts.0@my.csun.edu}}
 }
 \date{}
\maketitle

\begin{abstract}
Let $Q$ be a finite set of points in the plane. For any set $P$ of points in
the plane, $S_{Q}(P)$ denotes the number of similar copies of $Q$ contained
in $P$. For a fixed $n$, Erd\H{o}s and Purdy asked to determine the maximum
possible value of $S_{Q}(P)$, denoted by $S_{Q}(n)$, over all sets $P$ of $n$
points in the plane. We consider this
problem when $Q=\triangle$ is the set of vertices of an isosceles right
triangle. We give exact solutions when $n\leq9$, and
provide new upper and lower bounds for $S_{\triangle}(n)$.

\end{abstract}

\footnotetext[1]{Supported in part by CURM, BYU, and NSF(Award \#: DMS - 0636648)}

\section{Introduction}

In the 1970s Paul Erd\H{o}s and George Purdy \cite{D,E,F} posed the
question, \textquotedblleft Given a finite set of points $Q$, what is
the maximum number of similar copies $S_{Q}(n)$ that can be
determined by $n$ points in the plane?\textquotedblright. This problem remains open in general. However, there has been some progress regarding the order of magnitude of this maximum as a function of $n$.  Elekes and Erd\H{o}s \cite{B} noted that $S_{Q}\left(  n\right) \leq n\left(
n-1\right)  $ for any pattern $Q$ and they also gave a quadratic lower bound
for $S_{Q}(n)$ when $\left\vert Q\right\vert =3$ or when all the
coordinates of the points in $Q$ are algebraic numbers. They also
proved a slightly subquadratic lower bound for all other patterns
$Q$. Later, Laczkovich and Ruzsa \cite{LR97} characterized precisely
those patterns $Q$ for which $S_{Q}\left(  n\right) =\Theta(n^{2})$. In spite of this, the coefficient of the quadratic term is not known for any non-trivial pattern; it is not even known if $\lim_{n\rightarrow \infty}S_Q(n)/n^2$ exists!

Apart from being a natural question in Discrete Geometry, this
problem also arose in connection to optimization of algorithms
designed to look for patterns among data obtained from scanners,
digital cameras, telescopes, etc. (See \cite{Bra, BrMoPa, BP05} for
further references.)

Our paper considers the case when $Q$ is the set of vertices of an isosceles right triangle.
The case when $Q$ is the set of vertices of an equilateral triangle has been considered in
\cite{A}. To avoid redundancy, we refer to an isosceles right
triangle as an ${\mathsf{IRT}}$ for the remainder of the paper. We
begin with some definitions.
Let $P$ denote a finite set of points in the plane. We
define $S_{\triangle}(P)$ to be the number of triplets in $P$ that are the vertices of an ${\mathsf{IRT}}$. Furthermore, let
\[
S_{\triangle}(n)=\max_{|P|=n}S_{\triangle}(P).
\]
As it was mentioned before, Elekes and Erd\H{o}s established that $S_{\triangle}(n)=\Theta(n^2)$ and it is implicit from their work that $1/18 \leq \liminf_{n\rightarrow\infty}%
S_{\triangle}(n)/n^{2} \leq 1$. The main goal of this paper is to
derive improved constants that bound the function
$S_{\triangle}(n)/n^{2}$.  Specifically, in Sections \ref{sec:
lower} and \ref{sec: upper}, we prove the following result:
\begin{theorem}
\[
0.433064 < \liminf_{n\rightarrow\infty}%
\frac{S_{\triangle}(n)}{n^{2}}\leq\frac{2}{3} < 0.66667.
\]
\end{theorem}
We then proceed to determine, in Section \ref{sec: small cases}, the
exact values of $S_{\triangle }(n)$ when $3\leq n\leq9$.  Several
ideas for the proofs of these bounds come from the equivalent bounds
for equilateral triangles in \cite{A}.

\section{Lower Bound\label{sec: lower}}

We use the following definition.
For $z\in P$, let $R_{\pi/2}(z,P)$ be the $\pi/2$
counterclockwise rotation of $P$ with center $z$. Furthermore, let $\deg
_{\pi/2}(z)$ be the number of isosceles right triangles in $P$ such that $z$
is the right-angle vertex of the triangle.
If $z\in P$, then $\deg_{\pi/2}(z)$ can be computed by simply rotating our
point set $P$ by $\pi/2$ about $z$ and counting the number of points in the
intersection other than $z$. Therefore,
\begin{equation}
\deg_{\pi/2}(z)=|P\cap R_{\pi/2}(z,P)|-1. \label{degpi/2}%
\end{equation}
Due to the fact that an ${\mathsf{IRT}}$ has only one right angle, then
\[
S_{\triangle}(P) = \sum_{z\in P} \deg_{\pi/2}(z).
\]
That is, the sum computes the number of ${\mathsf{IRT}}$s in $P$. From this identity
an initial $5/12$ lower bound can be derived for $\liminf_{n\rightarrow \infty}S_\triangle(n)/n^2$ using the set
\[
P=\left\{  (x,y)\in\mathbb{Z}^{2}:0\leq x\leq \sqrt{n},0\leq y\leq \sqrt{n}\right\}.
\]
We now improve this bound.

The following theorem generalizes our method for finding a lower bound.  We denote by $\Lambda$ the lattice generated by the points
$(1,0)$ and $(0,1)$; furthermore, we refer to points in $\Lambda$ as \emph{lattice points}. The next result provides a formula for the leading term of $S_{\triangle}(P)$ when our
points in $P$ are lattice points enclosed by a given shape. This theorem, its
proof, and notation, are similar to Theorem 2 in \cite{A}, where the
authors obtained a similar result for equilateral triangles in place of
${\mathsf{IRT}}$s.

\begin{theorem}
\label{integral} Let K be a compact set with finite perimeter and area 1.
Define $f_{K}:\mathbb{C}\rightarrow\mathbb{R}^{+}$ as $f_{K}(z)=Area(K\cap
R_{\pi/2}(z,K))$ where $z\in K$. If $K_{n}$ is a similar copy of $K$
intersecting $\Lambda$ in exactly $n$ points, then
\[
S_{\triangle}(K_{n}\cap\Lambda)=\left(  \int_{K}f_{K}(z)\,dz\right)
n^{2}+O(n^{3/2}).
\]

\end{theorem}

\begin{proof}
Given a compact set $L$ with finite area and perimeter, we have that
\[
\left\vert rL\cap\Lambda\right\vert ={\mathrm{Area}(rL)}
+O(r)=r^{2}\mathrm{Area}(L)+O(r),
\]
where $rL$ is the scaling of $L$ by a factor $r$. Therefore,
\begin{align*}
S_{\triangle}(K_{n}\cap\Lambda)  &  =\sum_{z\in
K_{n}\cap\Lambda}|(\Lambda \cap K_{n})\cap
R_{\pi/2}(z,(K_{n}\cap \Lambda))|-1\\
&  =\sum_{z\in K_{n}\cap\Lambda}\mathrm{Area}(K_{n}\cap R_{\pi/2}%
(z,K_{n}))+O(\sqrt{n}).
\end{align*}
We see that each error term in the sum is bounded by the perimeter
of $K_{n}$, which is finite by hypothesis. Thus,
\begin{align*}
S_{\triangle}(K_{n}\cap\Lambda)  &  =n^{2}\sum_{z\in K_{n}\cap\Lambda}\frac
{1}{n^{2}}\mathrm{Area}(K_{n}\cap R_{\pi/2}(z,K_{n}))+O(n^{3/2})\\
&  =n^{2}\sum_{z\in K_{n}\cap\Lambda}\frac{1}{n}\mathrm{Area}(\frac{1}%
{\sqrt{n}}\left(  K_{n}\cap R_{\pi/2}(z,K_{n})\right)  )+O(n^{3/2})\\
&  =n^{2}\sum_{z\in K_{n}\cap\Lambda}\frac{1}{n}\mathrm{Area}\left(  \frac
{1}{\sqrt{n}}K_{n}\cap R_{\pi/2}\left(  \frac{z}{\sqrt{n}},\frac{1}{\sqrt{n}%
}K_{n}\right)  \right)  +O(n^{3/2})\text{.}%
\end{align*}
The last sum is a Riemann approximation for the function $f_{(1/\sqrt{n}%
)K_{n}}$ over the region $(1/\sqrt{n})K_{n}$, thus
\[
S_{\triangle}(K_{n}\cap\Lambda)=n^{2}\left(  \int_{\frac{1}{\sqrt{n}}K_{n}%
}f_{\frac{1}{\sqrt{n}}K_{n}}(z)\,dz+O\left(  \frac{1}{\sqrt{n}}\right)
\right)  +O(n^{3/2}).
\]
Since
\[
\mathrm{Area}\left(  \frac{1}{\sqrt{n}}K_{n}\right)  =\frac{1}{n}%
\mathrm{Area}(K_{n})=\frac{1}{n}(n+O(\sqrt{n}))=1+O\left(  \frac{1}{\sqrt{n}%
}\right)  =\mathrm{Area}(K)+O\left(  \frac{1}{\sqrt{n}}\right)  ,
\]
it follows that,
\[
\int_{\frac{1}{\sqrt{n}}K_{n}}f_{\frac{1}{\sqrt{n}}K_{n}}(z)\,dz=\int_{K}%
f_{K}(z)\,dz+O\left(  \frac{1}{\sqrt{n}}\right)  \text{.}%
\]
As a result,
\begin{align*}
S_{\triangle}(K_{n}\cap\Lambda)  &  =n^{2}\int_{\frac{1}{\sqrt{n}}K_{n}%
}f_{\frac{1}{\sqrt{n}}K_{n}}(z)\,dz+O(n^{3/2})\\
&  =n^{2}\int_{K}f_{K}(z)\,dz+O(n^{3/2})\text{.}\qedhere
\end{align*}

\end{proof}

The importance of this theorem can be seen immediately.  Although our $5/12$ lower bound for $\liminf_{n\rightarrow\infty} S_{\triangle}(n)/n^{2}$ was derived by summing the degrees of each point in a square lattice, the same result can be obtained by letting $K$ be the square $\{ (x,y):|x| \leq \frac{1}{2}, |y| \leq \frac{1}{2} \}$. It follows that $f_K(x,y)=(1-|x|-|y|)(1-||x|-|y||)$ and
\[
S_{\triangle}(K_{n}\cap\Lambda)=\left(  \int_{K}f_{K}(z)\,dz\right)
n^{2}+O(n^{3/2}) = \frac{5}{12}n^{2}+O(n^{3/2})\text{.}%
\]

An improved lower bound will follow
provided that we find a set $K$ such that the value for the
integral in Theorem \ref{integral} is larger than $5/12$. We get a larger value for the integral by letting $K$ be
 the circle $\{ z \in \mathbb{C} : |z|\leq 1/\sqrt{\pi} \}$. In this case
\begin{equation} \label{circle function}
f_K(z)=\frac{2}{\pi} \arccos (\frac{\sqrt{2 \pi}}{2} |z|)-|z|\sqrt{\frac{2}{\pi}-|z|^2}
\end{equation}
and
\[
S_{\triangle}(K_{n}\cap\Lambda)=\left(  \int_{K}f_{K}(z)\,dz\right)
n^{2}+O(n^{3/2}) = \left(  \frac{3}{4}-\frac{1}{\pi}\right)n^{2}+O(n^{3/2})\text{.}%
\]

It was conjectured in \cite{A} that not only does $\lim_{n\rightarrow \infty} E(n)/n^2$ exist, but it is attained by the uniform lattice in the shape of a circle.  ($E(n)$ denotes the maximum number of equilateral triangles determined by $n$ points in the plane.)  The corresponding conjecture in the case of the isosceles right triangle turns out to be false.  That is, if $\lim_{n\rightarrow \infty}S_{\triangle}(n)/n^2$ exists, then it must be strictly greater than $3/4-1/\pi$.  Define $\overline{\Lambda}$ to be the translation of $\Lambda$ by the vector $(1/2,1/2)$. The following lemma will help us to improve our lower bound.

\begin{lemma} \label{Lem: lambda}
If $(j,k)\in \mathbb{R}^2$ and $\Lambda^\prime=\Lambda$ or $\Lambda^\prime=\overline{\Lambda}$, then
\[
R_{\pi/2}((j,k),\Lambda^\prime) \cap \Lambda^\prime=\left\{
\begin{array}{l}
\Lambda^\prime \text{ if } (j,k)\in \Lambda \cup \overline{\Lambda}, \\
\varnothing \text{ else.}
\end{array}
\right.
\]

\end{lemma}

\begin{proof}
Observe that
\[R_{\pi/2} ((j,k),(s,t)) =
\begin{pmatrix}
0 & -1\\
1& 0\\
\end{pmatrix}
\begin{pmatrix}
s-j\\
t-k\\
\end{pmatrix}
+\begin{pmatrix}
j\\
k\\
\end{pmatrix}
=
\begin{pmatrix}
k-t+j\\
s-j+k\\
\end{pmatrix}.
\]
First suppose $(s,t)\in \Lambda$. Since $s,t\in \mathbb{Z}$, then $(k-t+j,s-j+k)\in \Lambda$ if and only if $k-j\in \mathbb{Z}$ and $k+j\in \mathbb{Z}$. This can only happen when either both $j$ and $k$ are half-integers (i.e., $(j,k)\in \overline{\Lambda}$), or both $j$ and $k$ are integers (i.e., $(j,k)\in \Lambda$).  Now suppose $(s,t)\in \overline{\Lambda}$.  In this case, because both $s$ and $t$ are half-integers, we conclude that $(k-t+j, s-j+k)\in \overline{\Lambda}$ if and only if both $k-j\in \mathbb{Z}$ and $k+j\in \mathbb{Z}$. Once again this occurs if and only if $(j,k)\in \Lambda \cup \overline{\Lambda}$.
\end{proof}

Recall that if $K$ denotes the circle of area 1, then $(3/4-1/\pi)n^2$ is the leading term of  $S_{\triangle}(K_n\cap \Lambda)$.  The previous lemma implies that, if we were to adjoin a point $z\in \mathbb{R}^2$ to $K_n\cap \Lambda$ such that $z$ has half-integer coordinates and is located near the center of the circle formed by the points of $K_n\cap \Lambda$, then $\deg_{\pi/2}(z)$ will approximately equal $|K_n\cap \Lambda|$.  We obtain the next theorem by further exploiting this idea.

\begin{theorem}
\[.43169\approx \frac{3}{4}-\frac{1}{\pi} < .433064 < \liminf_{n\rightarrow \infty}\frac{S_{\triangle}(n)}{n^2}\]
\end{theorem}

\begin{proof}
Let $K$ be the circle of area 1, $A= K_{m_1}\cap \Lambda$, and $B = K_{m_2}\cap\overline{\Lambda}$.  Moreover, position $B$ so that its points are centered on the circle formed by the points in $A$ (See Figure \ref{Fig: Centered}). We let $n=m_1+m_2=|A \cup B|$ and $m_2=x\cdot m_1$, where $0<x<1$ is a constant to be determined.
\begin{figure}[h]
\centering
\includegraphics{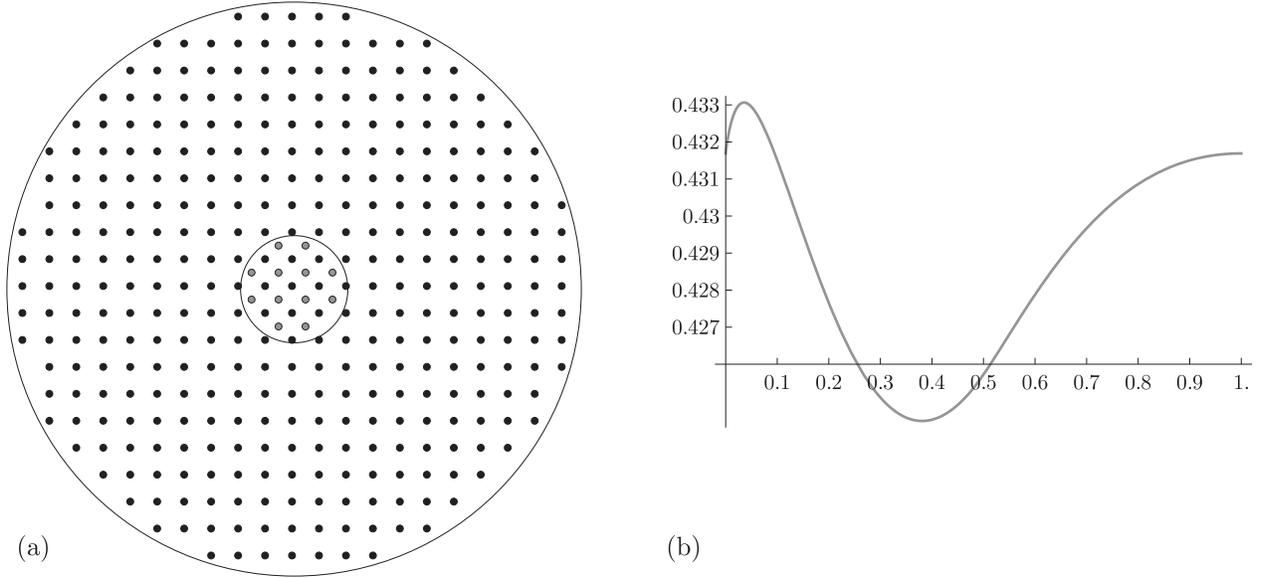}
\caption{(a) Set $B$ (gray points) centered on set $A$ (black
points),  (b) Plot of the $n^2$ coefficient of $S_\triangle(A \cup
B)$ as $x$ ranges from 0 to 1.} \label{Fig: Centered}
\end{figure}

We proceed to maximize the leading coefficient of
$S_{\triangle}(A\cup B)$ as $x$ varies from 0 to $1$.  By Lemma
\ref{Lem: lambda}, there cannot exist an \irt$\:$whose right-angle
vertex lies in $A$ while one $\pi/4$ vertex lies in $A$ and the
other lies in $B$.  Similarly, there cannot exist an \irt whose
right angle-vertex lies in $B$ while one $\pi/4$ vertex lies in $A$
and the other lies in $B$.  Therefore, each \irt with vertices in
$A\cup B$ must fall under one of the following four cases:
\bigskip
\\
\noindent \textit{Case 1: All three vertices in $A$}.  Using Theorem \ref{integral}, it follows that there are $(3/4 - 1/\pi)m_1^2 + O(m_1^{3/2})$ {\irt}s in this case. Since $m_1=n/(1+x)$, the number of {\irt}s in terms of $n$ equals
\begin{equation}\label{Case 1}
\left(\frac{3}{4} - \frac{1}{\pi}\right)\frac{n^2}{(1+x)^2} + O(n^{3/2}).
\end{equation}
\\
\noindent \textit{Case 2: All three vertices in $B$}.  By Theorem \ref{integral}, there are $(3/4 - 1/\pi)m_2^2 + O(m_2^{3/2})$ {\irt}s in this case.  This time $m_2=nx/(1+x)$ and the number of {\irt}s in terms of $n$ equals
\begin{equation}\label{Case 2}
\left(\frac{3}{4} - \frac{1}{\pi}\right)\frac{n^2x^2}{(1+x)^2} + O(n^{3/2}).
\end{equation}
\\
\noindent \textit{Case 3: Right-angle vertex in $B$, $\pi/4$ vertices in $A$}.  The relationship given by Lemma \ref{Lem: lambda} allows us to slightly adapt the proof of Theorem \ref{integral} in order to compute the number of {\irt}s in this case.  The integral approximation to the number of {\irt}s in this case is given by

\[\sum_{z\in K_{m_2}\cap \overline{\Lambda}} |(K_{m_1}\cap\Lambda)\cap R_{\pi/2}(z,(K_{m_1}\cap\Lambda))| = m_1^2\left(\int_{\frac{1}{\sqrt{m_1}}K_{m_2}} f_{\frac{1}{\sqrt{m_1}}K_{m_1}}(z)\,dz \right) + O(m_1^{3/2}).\]
But
\[\mathrm{Area}\left(\frac{1}{\sqrt{m_1}}K_{m_2}\right) = \mathrm{Area}\left(\sqrt{\frac{m_2}{m_1}}K\right) + O(\sqrt{m_1}),\]
so
\[m_1^2\left(\int_{\frac{1}{\sqrt{m_1}}K_{m_2}} f_{\frac{1}{\sqrt{m_1}}K_{m_1}}(z)\,dz \right) + O(m_1^{3/2}) = m_1^2\left(\int_{\sqrt{\frac{m_2}{m_1}}K} f_{K}(z)\,dz \right) + O(m_1^{3/2}).\]
Expressing this value in terms of $n$ gives
\begin{equation}\label{Case 3}
\left(\int_{\sqrt{x}K} f_{K}(z)\,dz \right)\frac{n^2}{(1+x)^2} + O(n^{3/2}).
\end{equation}
\\
\noindent \textit{Case 4: Right-angle vertex in $A$, $\pi/4$ vertices in $B$}.  As in Case 3, the number of {\irt}s is given by

\begin{equation}\label{Eqn: integral}
\sum_{z\in K_{m_1}\cap \Lambda} |(K_{m_2}\cap\overline{\Lambda})\cap R_{\pi/2}(z,(K_{m_2}\cap \overline{\Lambda}))| = m_2^2\left(\int_{\frac{1}{\sqrt{m_2}}K_{m_1}} f_{\frac{1}{\sqrt{m_2}}K_{m_2}}(z)\,dz \right) + O(m_2^{3/2}).
\end{equation}

Now recall that $f_{(1/\sqrt{m_2})K_{m_2}}(z) = \mathrm{Area}\left((1/\sqrt{m_2})K_{m_2}\cap R_{\pi/2}(z,(1/\sqrt{m_2})K_{m_2})\right)$.  It follows that $f_{(1/\sqrt{m_2})K_{m_2}}(z_0) = 0$ if and only if $z_0$ is farther than $\sqrt{2/\pi}$ from the center of $(1/\sqrt{m_2})K_{m_2}$.  Thus for small enough values of $m_2$, the region of integration in Equation (\ref{Eqn: integral}) is actually $(\sqrt{2/m_2})K_{m_2}$, so it does not depend on $m_1$. We consider two subcases.

First, if $x \leq 1/2$ (i.e., $m_2 \leq m_1/2$), then
\[\sqrt{\frac{2}{\pi}} = \frac{1}{\sqrt{m_2}}\frac{\sqrt{2m_2}}{\sqrt{\pi}} \leq \frac{1}{\sqrt{m_2}}\frac{\sqrt{2}}{\sqrt{\pi}}\frac{\sqrt{m_1}}{\sqrt{2}} = \frac{1}{\sqrt{m_2}}\sqrt{\frac{m_1}{\pi}}.\]
The left side of the above inequality is the radius of $(\sqrt{2/m_2})K_{m_2}$, meanwhile the right side is the radius of $(1/\sqrt{m_2})K_{m_1}$, thus the region of integration where $f_{\frac{1}{\sqrt{m_2}}K_{m_2}}$ is nonzero equals $(\sqrt{2/m_2})K_{m_2}$.  Hence, the number of {\irt}s equals

\begin{align}\nonumber
m_2^2\left(\int_{\sqrt{\frac{2}{m_2}}K_{m_2}} f_{\frac{1}{\sqrt{m_2}}K_{m_2}}(z)\,dz \right) + O(m_2^{3/2}) &= m_2^2\left(\int_{\sqrt{2}K} f_{K}(z)\,dz \right) + O(m_2^{3/2})\\ \label{Case 4A} &= \left(\int_{\sqrt{2}K} f_{K}(z)\,dz \right)x^2 n^2 + O(n^{3/2}).
\end{align}

Now we consider the case $x>1/2$ (i.e., $m_2> m_1/2$).  In this case, $f_{\frac{1}{\sqrt{m_2}}K_{m_2}}$ is nonzero for all points in $\frac{1}{\sqrt{m_2}}K_{m_1}$. Thus the number of {\irt}s in this case equals
\begin{align}\nonumber
m_2^2\left(\int_{\frac{1}{\sqrt{m_2}}K_{m_1}} f_{\frac{1}{\sqrt{m_2}}K_{m_2}}(z)\,dz \right) + O(m_2^{3/2}) &= m_2^2\left(\int_{\sqrt{\frac{m_1}{m_2}}K} f_{K}(z)\,dz \right) + O(m_2^{3/2})\\ \label{Case 4B}
&=\left(\int_{\sqrt{\frac{1}{x}}K} f_{K}(z)\,dz \right)\frac{n^2x^2}{(1+x)^2} + O(n^{3/2})
\end{align}

By Equation (\ref{circle function}), we have that for $t>0$,
\begin{align*}
\int_{tK}f_{K}(z)\,dz =&2\pi \int_{0}^{t/\sqrt{\pi }}\left( \frac{2}{\pi }
\arccos ( \frac{\sqrt{2\pi }}{2}r) -r\sqrt{\frac{2}{\pi }-r^{2}}
\right) r\,dr \\
=&\frac{1}{2\pi }\left( 4t^2\arccos ( \frac{t}{\sqrt{2}}) +2\arcsin
( \frac{t}{\sqrt{2}}) -t(t^{2}+1)\sqrt{2-t^{2}}\right).
\end{align*}
Therefore, putting all four cases together (i.e., expressions
(\ref{Case 1}), (\ref{Case 2}), (\ref{Case 3}), and either
(\ref{Case 4A}) or (\ref{Case 4B})), we obtain that the $n^2$
coefficient of $S_\triangle (A\cup B)$ equals
\begin{equation*}
\frac{1}{4\pi (x+1)^{2}}\left(8x\arccos \sqrt{\frac{x}{2}}+4\arcsin \sqrt{\frac{x%
}{2}}+(5\pi -4)x^{2}+ (3\pi -4)-2(x+1)\sqrt{2x-x^{2}}\right)
\end{equation*}
if $0<x\leq1/2$, or
\begin{multline*}
\frac{1}{4\pi (x+1)^{2}}\left(8x\left( \arccos \sqrt{\frac{x}{2}}+\arccos \sqrt{
\frac{1}{2x}}\right) +4\arcsin \sqrt{\frac{x}{2}}+4x^{2}\arcsin \sqrt{\frac{1
}{2x}}+ \right.\\ \left.
(3\pi -4)(x^{2}+1)-2(x+1)\left( \sqrt{2x-x^{2}}+\sqrt{2x-1}\right) \right)
\end{multline*}
if $1/2<x<1$. Letting $x$ vary from 0 to 1, it turns out that this
coefficient is maximized (see Figure \ref{Fig: Centered}) when $x\approx
.0356067$ (this corresponds to when the radius of $B$ is
approximately 18.87\% of the radius of $A$).  Letting $x$ equal this
value gives $0.433064$ as a decimal approximation to the maximum
value attained by the $n^2$ coefficient.
\end{proof}


At this point, one might be tempted to further increase the
quadratic coefficient by placing a third set of lattice points
arranged in a circle and centered on the circle formed by $B$.  It
turns out that forming such a configuration does not improve the
results in the previous theorem.  This is due to Lemma
\ref{Lem: lambda}.  More specifically, given our construction from
the previous theorem, there is no place to adjoin a point $z$ to the
center of $A\cup B$ such that $z\in \Lambda$ or $z\in
\overline{\Lambda}$.  Hence, if we were to add the point $z$ to the
center of $A\cup B$, then any new {\irt}s would have their
right-angle vertex located at $z$ with one $\pi/4$ vertex in $A$ and
the other $\pi/4$ vertex in $B$.  Doing so can produce at most $2m_2
= 2xm_1\approx .0712m_1$ new {\irt}s (recall that $x\approx
.0356066$ in our construction).  On the other hand, adding $z$ to
the perimeter of $A$, gives us $m_1f_{K}(1/\sqrt{\pi}) \approx
.1817m_1$ new {\irt}s.

\section{Upper Bound\label{sec: upper}}

We now turn our attention to finding an upper
bound for $S_{\triangle
}(n)/n^{2}$. It is easy to see that $S_{\triangle}(n)\leq n^{2}-n$,
since any pair of points can be the vertices of at most 6 ${\mathsf{IRT}}$s. Our next theorem
improves this bound. The idea is to
prove that there exists a point in $P$ that does not belong to many
${\mathsf{IRT}}$s. First, we need the following definition.

For every $z\in P$, let $R_{\pi/4}^{+}(z,P)$ and
$R_{\pi/4}^{-}(z,P)$ be the dilations of $P$, centered at $z$, by a factor of
$\sqrt{2}$ and $1/\sqrt{2}$, respectively; followed by a $\pi/4$
counterclockwise rotation with center $z$. Furthermore, let $\deg_{\pi/4}%
^{+}(z)$ and $\deg_{\pi/4}^{-}(z)$ be the number of isosceles right triangles
$zxy$ with $x,y\in P$ such that $zxy$ is ordered in counterclockwise order,
and $zy$, respectively $zx$, is the hypotenuse of the triangle $zxy$.

Much like the case of $\deg_{\pi/2}$, $\deg_{\pi/4}^{+}$ and $\deg_{\pi/4}%
^{-}$ can be computed with the following identities,
\[
\deg_{\pi/4}^{+}\left(  z\right)  =\left\vert P\cap R_{\pi/4}^{+}(z,P)\right\vert
-1\text{ and }\deg_{\pi/4}^{-}\left(  z\right)  =\left\vert P\cap R_{\pi/4}%
^{-}(z,P)\right\vert -1\text{.}%
\]

\begin{theorem}
\label{theorem3}For $n\geq3$,
\[
S_{\triangle}(n)\leq\left\lfloor \frac{2}{3}(n-1)^{2}-\frac{5}{3}\right\rfloor
.
\]

\end{theorem}

\begin{proof}
By induction on $n$. If $n=3$, then $S_{\triangle}(3)\leq1 = \left\lfloor \left(
2\cdot4-5\right)  /3\right\rfloor$. Now suppose the theorem holds for
$n=k$. We must show this implies the theorem holds for $n=k+1$. Suppose that
there is a point $z\in P$ such that $\deg_{\pi/2}(z)+\deg_{\pi/4}%
^{+}(z)+\deg_{\pi/4}^{-}(z)\leq\lfloor(4n-5)/3\rfloor$. Then by induction,
\begin{align*}
S_{\triangle}(k+1)  &  \leq\deg_{\pi/2}(z)+\deg_{\pi/4}^{+}(z)+\deg_{\pi
/4}^{-}(z)+S_{\triangle}(k)\\
&  \leq\left\lfloor \frac{4k-1}{3}\right\rfloor +\left\lfloor \frac{2}%
{3}(k-1)^{2}-\frac{5}{3}\right\rfloor =\left\lfloor \frac{2}{3}k^{2}-\frac
{5}{3}\right\rfloor .
\end{align*}
The last equality can be verified by considering the three possible residues
of $k$ when divided by 3. Hence, our theorem is proved if we can find a point
$z\in P$ with the desired property.

Let $x,y\in P$ be points such that $x$ and $y$ form the diameter of $P$. In
other words, if $w\in P$, then the distance from $w$ to any other point in $P$
is less than or equal to the distance from $x$ to $y$. We now prove that
either $x$ or $y$ is a point with the desired property mentioned above. We
begin by analyzing $\deg_{\pi/4}^{-}$. We use the same notation from Theorem 1 in \cite{A}.

Define $N_{x}=P \cap R_{\pi/4}^{-}(x,P)\backslash\{x\}$ and $N_{y}=P \cap R_{\pi/4}%
^{-}(y,P)\backslash\{y\}$. It follows from our identities that, $\deg_{\pi
/4}^{-}(x)=\vert N_{x}\vert $ and $\deg_{\pi/4}^{-}(y)=\vert
N_{y}\vert $. Furthermore, by the Inclusion-Exclusion Principle for
finite sets, we have $\vert N_{x}\vert +\vert N_{y}\vert =\vert N_{x}\cup
N_{y}\vert +\vert N_{x}\cap N_{y}\vert .$
We shall prove by contradiction that $|N_{x}\cap N_{y}|\leq1$. Suppose that there
are two points $u,v\in N_{x}\cap N_{y}$. This means that there are points
$u_{x},v_{x},u_{y},v_{y}\in P$ such that the triangles $xu_{x}u,xv_{x}%
v,yu_{y}u,yv_{y}v$ are ${\mathsf{IRT}}$s oriented counterclockwise with right
angle at either $u$ or $v$.

\begin{figure}[h]
\begin{center}
\includegraphics[scale=.91]{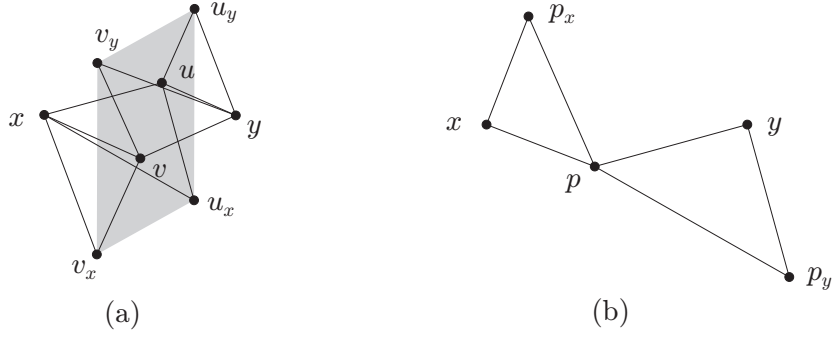}
\end{center}
\caption{Proof of Theorem 4.}%
\label{theoremlemma4}%
\end{figure}

But notice that the line segments $u_{x}u_{y}$ and $v_{x}v_{y}$ are simply the
$(\pi/2$)-counterclockwise rotations of $xy$ about centers $u$ and $v$
respectively. Hence, $u_{x}u_{y}v_{x}v_{y}$ is a parallelogram with two sides
having length $xy$ as shown in Figure \ref{theoremlemma4}(a). This is a
contradiction since one of the diagonals of the parallelogram is longer than
any of it sides. Thus, $|N_{x}\cap N_{y}|\leq1$. Furthermore, $x\notin N_{y}$ and $y\notin N_{x}$, so $|N_{x}\cup N_{y}|\leq n-2$ and thus
\[
\deg_{\pi/4}^{-}(x)+\deg_{\pi/4}^{-}(y)=\left\vert N_{x} \cup N_{y} \right\vert
+\left\vert N_x \cap N_{y}\right\vert \leq n-2+1=n-1\text{.}%
\]
This also implies that
\[
\deg_{\pi/4}^{+}(x)+\deg_{\pi/4}^{+}(y)\leq n-1,
\]
since we can follow the exact same argument applied to the reflection of $P$
about the line $xy$.

We now look at $\deg_{\pi/2}(x)$ and $\deg_{\pi/2}(y)$. First we need the
following lemma.

\begin{lemma}
For every $p\in P$, at most one of $R_{\pi/2}(x,p)$ or $R_{\pi/2}(y,p)$
belongs to $P$.
\end{lemma}

\begin{proof}
Let $p_{x}=R_{\pi/2}(x,p)$ and $p_{y}=R_{\pi/2}(y,p)$ (see Figure
\ref{theoremlemma4}(b)). Note that the distance $p_{x}p_{y}$ is exactly the
distance $xy$ but scaled by $\sqrt{2}$. This contradicts the fact that $xy$ is
the diameter of $P$.
\end{proof}

Let us define a graph $G$ with vertex set $V(G)=P\backslash\{x,y\}$ and where
$uv$ is an edge of $G$, (i.e., $uv\in E(G)$) if and only if $v=R_{\pi/2}(x,u)$
or $v=R_{\pi/2}(y,u)$.

\begin{lemma} \label{inequalitylemma}%
\[
0\leq\deg_{\pi/2}(x)+\deg_{\pi/2}(y)-|E(G)|\leq1\text{.}%
\]

\end{lemma}

\begin{proof}
The left inequality follows from the fact each edge counts an ${\mathsf{IRT}}$
in either $\deg_{\pi/2}(x)$ or $\deg_{\pi/2}(y)$ and possibly in both.
However, if $uv$ is an edge of $G$ so that $v=R_{\pi/2}(x,u)$ and $u=R_{\pi
/2}(y,v)$, then $xuyv$ is a square, so this can only happen for at most one edge.
\end{proof}

Now, let $\deg_{G}(u)$ be the number of edges in $E(G)$ incident to $u$.  We prove the following lemma.

\begin{lemma}
\label{deglemma} For every $u\in V(G)$, $\deg_{G}(u)\leq2$.
\end{lemma}

\begin{proof}
Suppose $uv_{1}\in E(G)$, see Figure \ref{Fig: lemma56}(a). Without loss of
generality we can assume that $u=R_{\pi/2}(y,v_{1})$. If $v_{3}=R_{\pi
/2}(y,u)\in P$, then we conclude that $xv_{3}>xy$ or $xv_1 >xy$, because $\angle xyv_{3}%
\geq\pi/2$ or $\angle xyv_1 \geq \pi/2$. This contradicts the fact that $xy$ is the diameter of $P$.
Similarly, if $v_{2}$ and $v_{4}$ are defined as $u=R_{\pi/2}(x,v_{4})$ and
$v_{2}=R_{\pi/2}(x,u)$, then at most one of $v_{2}$ or $v_{4}$ can be in $P$.
\end{proof}

\begin{figure}[h]
\begin{center}
\includegraphics[scale=.91]{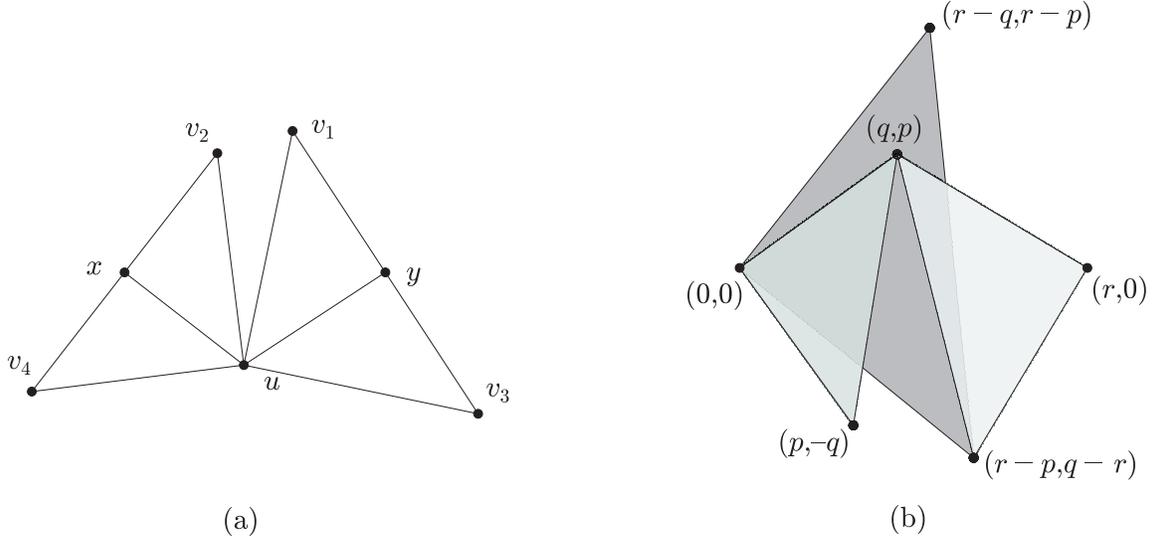}
\end{center}
\caption{Proof of Lemmas \ref{deglemma} and \ref{pathlemma}.}%
\label{Fig: lemma56}%
\end{figure}

We still need one more lemma for our proof.

\begin{lemma}
\label{pathlemma} All paths in $G$ have length at most 2.
\end{lemma}

\begin{proof}
We prove this lemma by contradiction. Suppose we can have a path of length 3
or more. To assist us, let us place our points on a cartesian coordinate
system with our diameter $xy$ relabeled as the points $(0,0)$ and $(r,0)$,
furthermore, assume $p,q\geq0$ and that the four vertices of the path of length $3$ are $(p,-q)$, $(q,p)$, $(r-p,q-r)$, and $(r-q,r-p)$. Our aim is to show that the distance between
$(r-q,r-p)$ and $(r-p,q-r)$ contradicts that $r$ is the diameter of $P$. Now,
if paths of length 3 were possible, then the distance between every pair of
points in Figure \ref{Fig: lemma56}(b) must be less than or equal to $r$.
Since $d((p,-q),(q,p))\leq r$ then $p^{2}+q^{2}\leq r^{2}/2$.

Now let us analyze the square of the distance from $(r-q,r-p)$ to $(r-p,q-r)$. Because $2(p^2+q^2)\geq (p+q)^2$, it follows that
\begin{align*}
d^{2}((r-q,r-p),(r-p,q-r))  &  =(-q+p)^{2}+(2r-p-q)^{2}\\
&  =4r^{2}-4r(p+q)+2(p^{2}+q^{2})\\
& \geq 4r^{2}-4\sqrt{2}r\sqrt{p^2+q^2}+2(p^{2}+q^{2})=\left(2r-\sqrt{2(p^2+q^2)}\right)^2.
\end{align*}
But $\sqrt{2(p^2+q^2)} \leq r$, so $(2r-\sqrt{2(p^2+q^2)})\geq r$ and thus
\[
d^{2}((r-q,r-p),(r-p,q-r))\geq r^{2}.
\]
Equality occur if and only if
$p=r/2$ and $q=r/2$; otherwise, $d((r-q,r-p),(r-p,q-r))$ is strictly
greater than $r$, contradicting the fact that the diameter of $P$ is $r$.
Therefore if $p\neq r/2$ or $q\neq r/2$ then there is no path of length
3. In the case that $p=r/2$ and $q=r/2$ the points $(q,p)$ and $(r-q,r-p)$ become the same and so do the points $(p,-q)$ and $(r-p,q-r)$. Thus we are left with a path of length 1.
\end{proof}

It follows from Lemmas \ref{deglemma} and \ref{pathlemma} that all paths of
length 2 are disjoint. In other words, $G$ is the union of disjoint paths of
length less than or equal to 2. Let $a$ denote the number of paths of length 2
and $b$ denote the number of paths of length 1, then
\[
\left\vert E(G)\right\vert =2a+b\text{ \textup{and} }3a+2b\leq n-2.
\]
Recall from Lemma \ref{inequalitylemma} that either $\deg_{\pi/2}(x)+\deg_{\pi/2}(y)=\left\vert E(G)\right\vert$ or $\deg_{\pi/2}(x)+\deg_{\pi/2}(y)=\left\vert E(G)\right\vert + 1$.  If $\deg_{\pi/2}(x)+\deg_{\pi/2}(y)=\left\vert E(G)\right\vert$, then
\[
2\left\vert E(G)\right\vert =4a+2b\leq n-2+a\leq n-2+\frac{n-2}{3},\]
so $\deg_{\pi/2}(x)+\deg_{\pi/2}(y)=\left\vert E(G)\right\vert \leq\frac{2}{3}\left(  n-2\right).$
  Moreover, if $\deg_{\pi/2}(x)+\deg_{\pi/2}(y)=\left\vert E(G)\right\vert +1$,
then $b\geq1$ and we get a minor improvement,
\[
2\left\vert E(G)\right\vert =4a+2b\leq n-2+a\leq n-4+\frac{n-2}{3},
\]
so $\deg_{\pi/2}(x)+\deg_{\pi/2}(y)=\left\vert E(G)\right\vert +1\leq \left(  2n-7\right)/3 < \frac
{2}{3}\left(  n-2\right) $.

We are now ready to put everything together. Between the two points $x$ and
$y$, we derived the following bounds:
\begin{align*}
\deg_{\pi/2}(x)+\deg_{\pi/2}(y)  &  \leq\frac{2}{3}(n-2),\\
\deg_{\pi/4}^{+}(x)+\deg_{\pi/4}^{+}(y)  &  \leq(n-1)\text{, and}\\
\deg_{\pi/4}^{-}(x)+\deg_{\pi/4}^{-}(y)  &  \leq(n-1)\text{.}%
\end{align*}
Because the degree of a point must take on an integer value, it must be the
case that either $x$ or $y$ satisfies $\deg_{\pi/2}+\deg_{\pi/4}^{+}+\deg
_{\pi/4}^{-}\leq\left\lfloor (4n-5)/3\right\rfloor $.
\end{proof}

\section{Small Cases\label{sec: small cases}}

In this section we determine the exact values of $S_{\triangle}(n)$ when $3\leq
n\leq9$.

\begin{theorem}
\label{smallcases} For $3\leq n\leq9$, $S_{\triangle}(3)=1$, $S_{\triangle
}(4)=4$, $S_{\triangle}(5)=8$, $S_{\triangle}(6)=11$, $S_{\triangle}(7)=15$,
$S_{\triangle}(8)=20$, and $S_{\triangle}(9)=28$.
\end{theorem}

\begin{figure}[h]
\begin{center}
\includegraphics[
height=2.75in ] {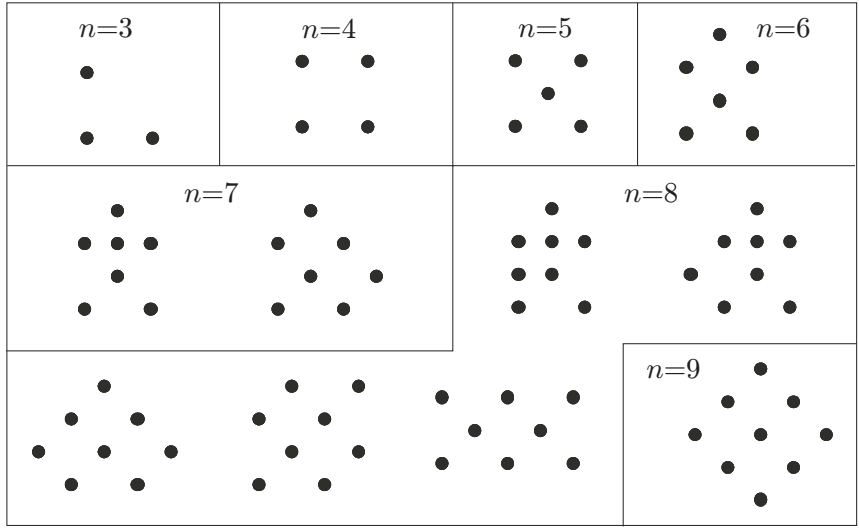}
\caption{Optimal sets achieving equality for $S_{\triangle}(n)$.}%
\label{Fig: optimal}
\end{center}
\end{figure}

\begin{proof}
We begin with $n=3$. Since $3$ points uniquely determine a triangle, and there is an \irt$\:$with 3 points (Figure \ref{Fig: smallcases}(a)), this
situation becomes trivial and we therefore conclude that
$S_{\triangle}(3)=1.$

Now let $n=4$. In Figure \ref{Fig: smallcases}(b) we show a
point-set $P$ such that $S_{\triangle}(P)=4$. This implies that
$S_{\triangle}(4)\geq4$. However, $S_{\triangle}(4)$ is also bounded
above by $\tbinom{4}{3}=4$. Hence, $S_{\triangle}(4)=4$.

To continue with the proof for the remaining values of $n$, we need the
following two lemmas.

\begin{lemma}
\label{p=4} Suppose $|P|=4$ and $S_{\triangle}(P)\geq2$. The sets in Figure
\ref{Fig: smallcases}(b)--(e), not counting symmetric repetitions, are the
only possibilities for such a set $P$.
\end{lemma}

\begin{proof}
Having $S_{\triangle}(P)\geq2$ implies that we must always have more
than one ${\mathsf{IRT}}$ in $P$. Hence, we can begin with a single
${\mathsf{IRT}}$ and examine the possible ways of adding a point and
producing more ${\mathsf{IRT}}$s. We accomplish this task in Figure
\ref{Fig: smallcases}(a). The 10 numbers in the figure indicate the
location of a point, and the total number of ${\mathsf{IRT}}$s after
its addition to the set of black dots. All other locations not
labeled with a number do not increase the number of
${\mathsf{IRT}}$s. Therefore, except for symmetries, all the
possibilities for $P$ are shown in Figures \ref{Fig:
smallcases}(b)--(e).
\end{proof}

\begin{figure}[ph]
\centering \includegraphics[height = 7.5in]{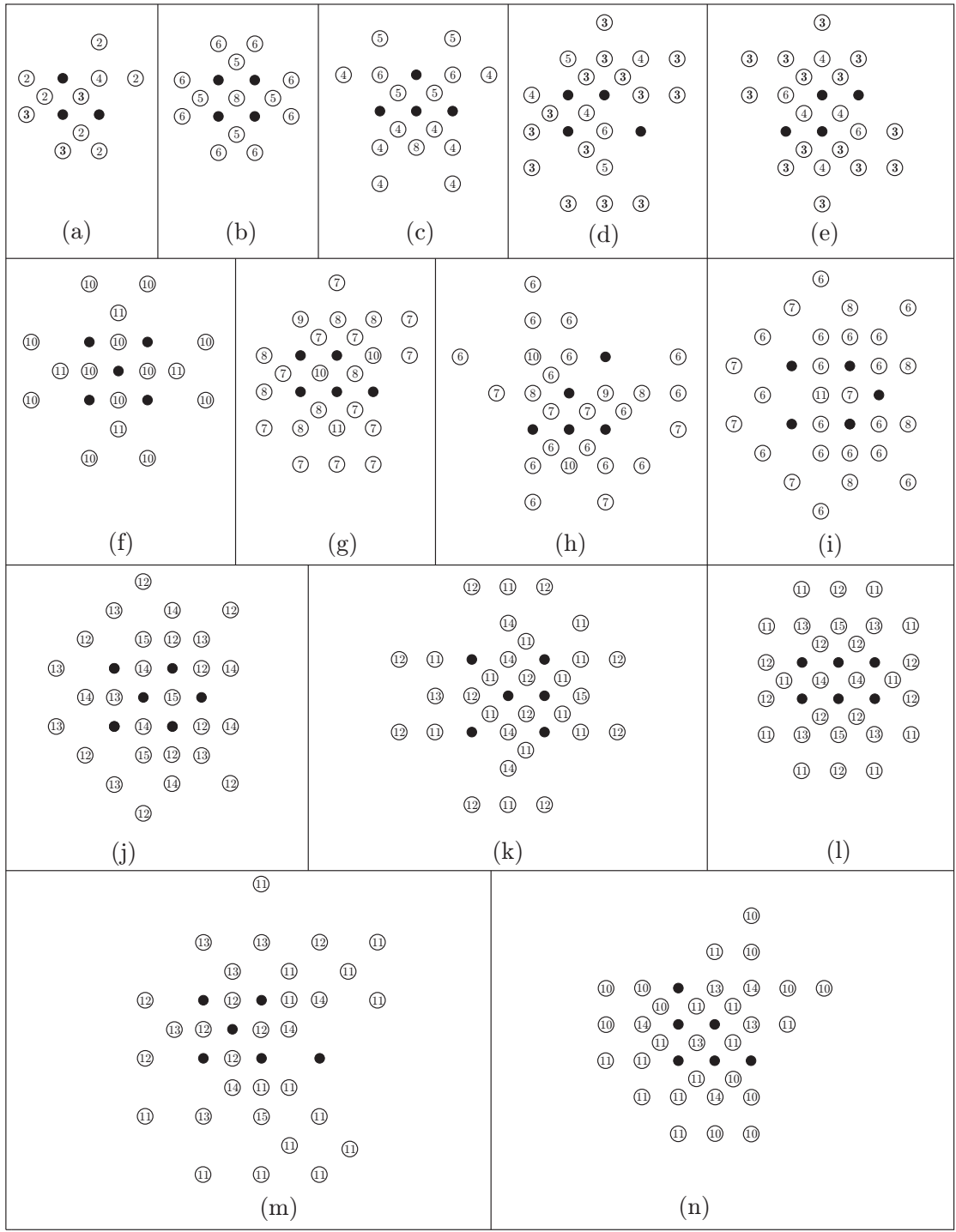}
\caption{{}Proof of Theorem 5. Each circle with a number indicates
the location of a point and the total number of ${\mathsf{IRT}}$s
resulting from its addition to the base set of black dots.}%
\label{Fig: smallcases}%
\end{figure}

\begin{lemma}
\label{sumlemma} Let $P$ be a finite set with $|P|=n$. Suppose that
$S_{\triangle}(A)\leq b$ for all $A\subseteq P$ with $|A|=k$. Then
\[
S_{\triangle}(P)\leq\left\lfloor \frac{n\left(  n-1\right)  \left(
n-2\right)  b}{k\left(  k-1\right)  \left(  k-2\right)  }\right\rfloor .
\]

\end{lemma}

\begin{proof}
Suppose that within $P$, every $k$-point configuration contains at most $b$
${\mathsf{IRT}}$s. The number of ${\mathsf{IRT}}$s in $P$ can then be counted
by adding all the ${\mathsf{IRT}}$s in every $k$-point subset of $P$. However,
in doing so, we end up counting a fixed ${\mathsf{IRT}}$ exactly $\tbinom
{n-3}{k-3}$ times. Because $S_{\triangle}(A)\leq b$ we get,
\[
\binom{n-3}{k-3}S_{\triangle}(P) = \sum_{\substack{A\subseteq P\\\left\vert A\right\vert
=k}}S_{\triangle}(A)\leq\binom{n}{k}b.
\]
Notice that $S_{\triangle}(P)$ can only take on integer values so,
\[
S_{\triangle}(P)\leq\left\lfloor \frac{\binom{n}{k}b}{\binom{n-3}{k-3}%
}\right\rfloor =\left\lfloor \frac{n\left(  n-1\right)  \left(  n-2\right)
b}{k\left(  k-1\right)  \left(  k-2\right)  }\right\rfloor .\qedhere
\]
\end{proof}

Now suppose $|P|=5$. If $S_{\triangle}(A)\leq1$ for all $A\subseteq P$ with
$|A|=4$, then by Lemma \ref{sumlemma}, $S_{\triangle}(P)\leq2$. Otherwise, by
Lemma \ref{p=4}, $P$ must contain one of the 4 sets shown in Figures
\ref{Fig: smallcases}(b)--\ref{Fig: smallcases}(e). The result now follows by examining the
possibilities for producing more ${\mathsf{IRT}}$s by placing a fifth point in
the 4 distinct sets. In Figures \ref{Fig: smallcases}(b),
\ref{Fig: smallcases}(c), \ref{Fig: smallcases}(d), and \ref{Fig: smallcases}%
(e) we accomplish this task. In the same way as we did in Lemma \ref{p=4},
every number in a figure indicates the location of a point, and the total number of
${\mathsf{IRT}}$s after its addition to the set of black dots. It
follows that the maximum value achieved by placing a fifth point is $8$ and so
$S_{\triangle}(5)=8$. The point-set that uniquely achieves equality is shown
in Figure \ref{Fig: smallcases}(f). Moreover, there is exactly one set $P$
with $S_{\triangle}(P)=6$ (shown in Figure \ref{Fig: smallcases}(g)), and two sets $P$ with $S_{\triangle}(P)=5$
(Figures \ref{Fig: smallcases}(h) and \ref{Fig: smallcases}(i)).

Now suppose $|P|=6$. If $S_{\triangle}(A)\leq4$ for all $A\subseteq P$ with
$|A|=5$, then by Lemma \ref{sumlemma}, $S_{\triangle}(P)\leq8$. Otherwise, $P$
must contain one of the sets in Figures \ref{Fig: smallcases}(f)--\ref{Fig: smallcases}(i). We now check all possibilities for adding more
${\mathsf{IRT}}$s by joining a sixth point to our 4 distinct sets. This is
shown in Figures \ref{Fig: smallcases}(f)--\ref{Fig: smallcases}(i). It follows that the maximum
value achieved is $11$ and so $S_{\triangle}(6)=11$. The point-set that
uniquely achieves equality is shown in Figure \ref{Fig: smallcases}(j). Also,
except for symmetries, there are exactly 3 sets $P$ with $S_{\triangle}(P)=10$
(Figures \ref{Fig: smallcases}(k)--\ref{Fig: smallcases}(m)) and only one set $P$ with
$S_{\triangle}(P)=9$ (Figure \ref{Fig: smallcases}(n)).

Now suppose $|P|=7$. If $S_{\triangle}(A)\leq8$ for all $A\subseteq P$ with
$|A|=6$, then by the Lemma \ref{sumlemma}, $S_{\triangle}(P)\leq14$.
Otherwise, $P$ must contain one of the sets in Figures \ref{Fig: smallcases}%
(j)--\ref{Fig: smallcases}(n). We now check all possibilities
for adding more ${\mathsf{IRT}}$s by joining a seventh point to our 5 distinct
configurations. We complete this task in Figures \ref{Fig: smallcases}%
(j)--\ref{Fig: smallcases}(n). Because the maximum value achieved is $15$, we deduce that
$S_{\triangle}(7)=15$. In this case, there are exactly two point-sets that
achieve 15 ${\mathsf{IRT}}$s.

The proof for the values $n=8$ and $n=9$ follows along the same
lines, but there are many more intermediate sets to be considered.
We omit the details. All optimal sets are shown in Figure \ref{Fig:
optimal}.
\end{proof}

Inspired by our method used to prove exact values of
$S_\triangle(n)$, a computer algorithm was devised to construct the
best 1-point extension of a given base set.  This algorithm,
together with appropriate heuristic choices for some initial sets,
lead to the construction of point sets with many ${\mathsf{IRT}}$s
giving us our best lower bounds for $S_\triangle(n)$ when $10\leq
n\leq25$.  These lower bounds are shown in Table 1 and the
point-sets achieving them in Figure \ref{Fig: bestconst}.

\begin{table}[ph] \centering

\begin{tabular}{||c||c|c|c|c|c|c|c|c||}
\multicolumn{9}{c}{}\\
\hline
$n$ & 10 & 11 & 12 & 13 & 14 & 15 & 16 & 17\\
\hline
$S_{\triangle}(n)\geq$ & 35 & 43 & 52 & 64 & 74 & 85 & 97 & 112\\
\hline
\multicolumn{9}{c}{}\\
\hline
 $n$ & 18 & 19 & 20 & 21 & 22 & 23 & 24 & 25\\
\hline
$S_{\triangle}(n)\geq$ & 124 & 139 & 156 & 176 & 192 & 210 & 229 & 252\\
\hline
\end{tabular}
\caption{Best lower bounds for
$S_\triangle (n).$\label{tab: small values}}%
\end{table}

\begin{figure}
[pht]
\begin{center}
\includegraphics[
height=1.69in
]%
{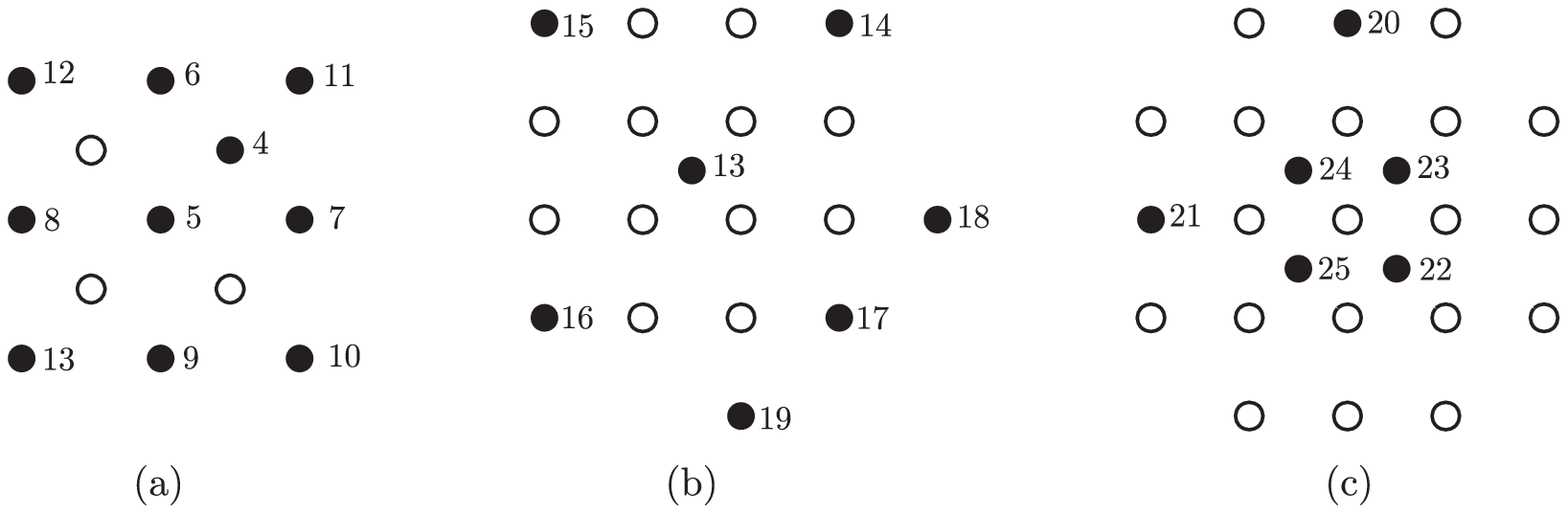}%
\caption{Best constructions $A_{n}$ for $n\leq25$. Each set $A_{n}$
is obtained as the union of the starting set (in white) and the
points with label $\leq n$. The value $S_{\triangle}(A_n)$ is given by
Table
\ref{tab: small values}.}%
\label{Fig: bestconst}%
\end{center}
\end{figure}

\bigskip
\noindent \textbf{Acknowledgements.} We thank Virgilio Cerna who, as
part of the CURM mini-grant that supported this project, helped to
implement the program that found the best lower bounds for smaller
values of $n$. We also thank an anonymous referee for some useful
suggestions and improvements to the presentation.

\end{document}